\documentclass{amsart}

\usepackage{amssymb}
\usepackage{amsmath}%
\usepackage{amsfonts}%
\usepackage{hyperref} 
\usepackage{url}
\newtheorem{thm}{Theorem}[section]
\newtheorem{prop}[thm]{Proposition}

\newtheorem{conj}[thm]{Conjecture}

\theoremstyle{definition}

\newcommand{\Q}{\mathbb{Q}}
\newcommand{\Z}{\mathbb{Z}}

\newcommand{\Pp}{\mathbb{P}}

\newcommand{\kk}{\bar{k}} 
\newcommand{\Cc}{\mathcal{C}}

\newcommand{\Hc}{\mathcal{H}}
\newcommand{\Uc}{\mathcal{U}}
\newcommand{\Vc}{\mathcal{V}}
\newcommand{\Wc}{\mathcal{W}}
\newcommand{\Xc}{\mathcal{X}}
\newcommand{\Cb}{\mathbf{C}} 
\newcommand{\Ll}{\mathcal{L}}
\newcommand{\KK}{\mathcal{K}}

\newcommand{\ba}{\mathbf{a}}

\DeclareMathOperator{\rk}{rank}
\DeclareMathOperator{\End}{End} 
\DeclareMathOperator{\Sym}{Sym} 

\title[Rank of Jacobians of $y^s=x(ax^r+b)$]{Rank of Jacobian Varieties of Curves  $y^s=x(ax^r+b)$} 
\author{Sajad Salami}
\address{Institute of Mathematics and Statistics, Rio de Janeiro State University, Maracanã, Rio de Janeiro, 20950-000, RJ, Brazil}
\email{sajad.salami@ime.uerj.br}

\date{\today} 


\begin{document}
	
\begin{abstract}
	Let $k$ be a number field. We investigate the Mordell-Weil ranks of Jacobian varieties $J_C$ associated with algebraic curves $C$ of genus $g \geq 1$ defined by affine equations of the form $y^s=x(ax^r+b)$, where $a, b \in k$ ($ab \neq 0$), and $r \geq 1, s \geq 2$ are fixed integers. Assuming the strong version of Lang's conjecture concerning rational points on varieties of general type, we establish that the ranks $r(J_C(k))$ are uniformly bounded as $C$ varies within this family. 
	Our methodology builds upon the geometric approach employed by H. Yamagishi and subsequently adapted by the author for the family $y^s=ax^r+b$. We construct a parameter space $\Wc_n$ for curves possessing $n+1$ specified rational points and analyze its birational model $\Xc_n$, a complete intersection variety. The geometric properties of the fibers of $\Xc_n \to \Sym^{n+1}(\Pp^1)$, specifically their genus and gonality, are studied. Combining these geometric insights with Faltings' theorem, uniformity conjectures stemming from Lang's work, and recent results connecting rank with the number of rational points, we deduce the main boundedness result. In the case of genus one curves $C$, it 
	states  that the rank of elliptic curves  $y^2=x (x^2+B)$ 
	is  uniformly bounded subject to the strong version of Lang's conjecture.
\end{abstract}
	
\maketitle
	
	\section{Introduction}
	
	A central theme in arithmetic geometry concerns the structure of the group of rational points on algebraic varieties defined over number fields. For an abelian variety $A$ over a number field $k$, the celebrated Mordell-Weil theorem asserts that the group $A(k)$ is finitely generated \cite{Silverman1986}. Thus, $A(k) \cong A(k)_{\text{tors}} \oplus \mathbb{Z}^r$, where $A(k)_{\text{tors}}$ is the finite torsion subgroup and $r = \rk A(k)$ is the Mordell-Weil rank. Understanding the behavior of this rank, particularly for Jacobians $J_C$ of algebraic curves $C$, remains a profound challenge. A fundamental open question asks whether $r(J_C(k))$ is uniformly bounded as $C$ varies over all curves of a fixed genus $g \geq 1$ defined over $k$.

    Even for elliptic curves ($g=1$) over $k=\Q$, this question is unresolved. While computational evidence reveals curves of high rank, theoretical heuristics suggest boundedness, possibly with rank exceeding 21 being rare \cite{Park2019}. Deep conjectures in Diophantine geometry provide conditional evidence for rank boundedness. Notably, Lang's conjectures on rational points on varieties of general type \cite{Lang1986, Lang1991} imply, via the work of Caporaso, Harris, and Mazur \cite{Caporaso1997, caporaso2022update}, strong uniformity statements about the number of rational points on curves. Furthermore, Pasten demonstrated that conjectures on Diophantine approximation imply Honda's conjecture on rank boundedness for elliptic curves \cite{Pasten2019}. The link between point counts and rank was solidified by Dimitrov, Gao, and Habegger \cite{Dimitrov2021}, making uniform point boundedness equivalent to uniform rank boundedness.

    An alternative approach, pioneered by Yamagishi \cite{Yamagishi2003} for the family $y^2=ax^4+bx^2+c$, employs geometric methods. By constructing parameter spaces for curves with multiple rational points and demonstrating that these spaces (or relevant subvarieties) become varieties of general type, one can invoke Lang's conjectures to deduce finiteness and uniformity properties, ultimately leading to rank bounds. This geometric strategy was adapted by the author in \cite{Salami2025toappear} to study the uniformity of  Mordell-Weil rank of Jacobian variety of  the family of curves
     $ y^s = ax^r+b$.

    This paper applies the same geometric methodology to investigate the closely related two-parameter family of curves defined by the affine equation
	\begin{equation} \label{curve_main}
		C_{a,b}: y^s=x(ax^r+b),
	\end{equation}
	where $r \geq 1, s \geq 2$ are fixed integers, and $k$ is a number field containing $\zeta_s$, a primitive $s$-th root of unity. We focus on smooth curves $C_{a,b}$ with genus $g_{r,s}(C_{a,b}) \geq 1$. Our main result, conditional on Lang's conjecture, establishes uniform rank boundedness for this family.

	\begin{thm} \label{main_theorem}
		Let $k$ be a number field containing a primitive $s$-th root of unity. Assume the strong version of Lang's conjecture (Conjecture \ref{conj1}) holds for varieties defined over $k$. Then the rank $r(J_C(k))$ is uniformly bounded as $J_C$ varies over the Jacobian varieties of all smooth curves $C$ defined over $k$ by equation \eqref{curve_main} with genus $g_{r,s}(C) \geq 1$.
	\end{thm}

 The main  result and its consequences provide new geometric evidence for rank boundedness  conjectures  for elliptic curves  similar to  that given in author's previous work \cite{Salami2025toappear} 
different from that of given in \cite{Pasten2019}. 
To be explicit, let us consider
the case $g_{r,s}(C) = 1$ and $k=\Q$, where the curve $C $ is an elliptic curves of the form  $y^2= x(x^2+ b)$ with  $\Z$.
By  Theorem\ref{main_theorem}, one can conclude that the  rank of elliptic curves  
 is  uniformly bounded subject to the strong version of Lang's conjecture,   see Section \ref{sec:applications}.
    We note that  the record of rank in this family is 14 due to  M. Watkins obtained in 2002, according to \cite{DujellaRankRecord}.
 As a particular case of the above family, is the congruent number elliptic curves. There are only 27 congruent number $N$ for which the rank of corresponding elliptic curves $y^2=x (x^2-N^2)$ is equal to 7. It is also suspected that there does not exist any congruent number elliptic  curve of rank 8, see   \cite{Watkins2014}. 
Theorem\ref{main_theorem}  implies the uniformity of rank  in family of congruent number elliptic  curves, subject to the strong version of Lang's conjecture,   see Section \ref{sec:applications}.

	The proof of Theorem\ref{main_theorem} relies crucially on studying the set $\Cc_{\ba_n}(k)$ consisting of smooth curves $C$ in the family \eqref{curve_main} that pass through a given set of $n+1$ points with $x$-coordinates $\alpha_i \in k$ ($i=0,\dots,n$) where $\ba_n$ represents the collection $\{\alpha_0, \dots, \alpha_n\}$ of distinct, non-zero elements of $k$ satisfying $\alpha_i^r \neq \alpha_j^r$ for $i \neq j$. We establish the following finiteness and uniformity results for these sets.

	\begin{thm} \label{finiteness_theorem}
		Let $k$ be a number field containing $\zeta_s$. Let $n_0=4$ if $s=2$ and $n_0=3$ otherwise. For any choice of $\ba_n = \{\alpha_0, \ldots, \alpha_n\}$ consisting of $n+1$ distinct, non-zero elements in $k$ with $\alpha_i^r \neq \alpha_j^r$ for $i \neq j$, and for integers $r \geq 1, s \geq 2$ such that $g_{r,s}(C) \geq 1$, the set $\Cc_{\ba_n}(k)$ is infinite if $n < n_0$ and finite if $n \geq n_0$. Consequently, for a fixed $\ba_n$ with $n \geq n_0$, the rank $r(J_C(k))$ is bounded for $C \in \Cc_{\ba_n}(k)$.
	\end{thm}
	
We use the uniformity results given this non-conditional result to discuss  a solution for the problem of arithmetic progression  on the points on the elliptic curves $y^2=x (x^2+B)$ as well as the congruent number elliptic curves, in Section~\ref{sec:applications}.

	Using the uniformity theorems on the number of $k$-rational points on curve of genus $\geq 2$ gibem in \cite{Caporaso1997}, as
	consequences of weak version of Lang's conjecture, we showed the cardinal number of $\Cc_{\ba_n}(k)$ does not depends on the choice of ${\ba_n}$ and  is   uniformly bounded and eventually empty for large  n.
	
	\begin{thm} \label{uniformity_theorem}
		Assume the weak version of Lang's conjecture (Conjecture \ref{conj1}) holds over $k$. Let $r \geq 1, s \geq 2$ such that $g_{r,s}(C) \geq 1$, and let $n_0$ be as in Theorem \ref{finiteness_theorem}. Then for $n \geq n_0$, there exists a uniform bound $M_0 = M_0(k, r, s, n)$ such that $\#\Cc_{\ba_n}(k) < M_0$ for all valid choices of $\ba_n$. Furthermore, if $g_{r,s}(C) \geq 2$, there exists an integer $m_0 > n_0$ such that $\Cc_{\ba_m}(k) = \emptyset$ for all $m \geq m_0$ and all valid $\ba_m$. (The strong Lang conjecture implies $M_0$ and $m_0$ can be chosen independently of $k$.)
	\end{thm}
	
	The structure of the paper mirrors the logical progression of the proof. Section \ref{sec_geom_constr} details the construction of the parameter space $\Wc_n$ using twists and its birational model $\Xc_n$, a complete intersection variety, establishing their relationship. Section \ref{sec_fiber_analysis} focuses on the geometry of the fibers $\Xc_{\ba_n}$ of $\Xc_n \to \Sym^{n+1}(\Pp^1)$, calculating their genus and gonality, analyzing rational points in low-genus cases, and invoking the theory of towers of curves for the general case, integrating necessary background results. Section \ref{sec_finiteness_uniformity_proofs} provides proofs for Theorems \ref{finiteness_theorem} and \ref{uniformity_theorem}, incorporating Faltings' theorem and Lang's conjectures. In Section \ref{sec_main_theorem_proof}, we prove  the main result, Theorem \ref{main_theorem}, using the uniformity theorem and the work of Dimitrov, Gao, and Habegger.	
	 Throughout, we cite analogous arguments from \cite{Salami2025toappear} where applicable.
	In Section~\ref{sec:applications}, we  have provided consequences of our results for   elliptic curves of the form $y^2=x(x^2+B)$ and  in particular the congruent number elliptic curves. 
	 Finally, in   Section \ref{examples}, we have provided a computational example of the correspondence between
	the set of curves  $\Cc_{\ba_n}$ and rational points on the fiber $\Xc_{\ba_n}$.

	\section{Some conjectures and results in Diophantine Geometry}
	\label{res-conjs}
	A smooth projective variety $X$ over $\kk$ is of \textit{general type} if its Kodaira dimension $\kappa(X)$ equals its dimension.
For curves, this corresponds to genus $g \geq 2$. The seminal result in this area is Faltings' theorem, formerly the Mordell conjecture, which states that: 
		{\it If $C$ is a smooth algebraic curve of genus $g \geq 2$ over a number field $k$, then $\#C(K) < \infty$ for any finite extension $K$ of $k$. }
	
	Lang proposed a far-reaching generalization of this to higher dimensions.
	\begin{conj}[Lang, \cite{Lang1986, Lang1991}]
		\label{conj1}
		Let $k$ be a number field and $X$ a smooth variety of general type over $k$.
		\begin{enumerate}
			\item[(a)] \textbf{(Weak)} The set of $k$-rational points $X(k)$ is not Zariski dense in $X$.
			\item[(b)] \textbf{(Strong)} There exists a proper Zariski-closed subset $Z \subset X$ such that for any finite extension $K$ of $k$, the set of $K$-rational points on $X \setminus Z$ is finite.
		\end{enumerate}
	\end{conj}
	
	Caporaso, Harris, and Mazur showed that Lang's conjecture has profound uniformity implications for rational points on curves \cite{Caporaso1997, caporaso2022update}, based on above conjectures.
	\begin{thm}[Uniformity I, \cite{Caporaso1997}]
		\label{UB}
		The weak version of Lang's conjecture implies that for every number field $k$ and integer $g \geq 2$, there exists a number $N(k,g)$ such that no curve of genus $g$ defined over $k$ has more than $N(k,g)$ rational points.
	\end{thm}
	
	\begin{thm}[Uniformity II, \cite{Caporaso1997}]
		\label{UGB}
		The strong version of Lang's conjecture implies that for any integer $g \geq 2$, there exists a number $N(g)$ such that for every number field $k$, there are only finitely many curves $C$ of genus $g$ over $k$ with $\#C(k) > N(g)$.
	\end{thm}
	
	 The following geometric theorem, which produces varieties of general type from families of curves, has an important role in the proof of the above uniformity results.
	\begin{thm}[Correlation Theorem, \cite{Caporaso1997}]
		\label{corel}
		Let $f: \Cc \rightarrow \mathcal{B}$ be a proper morphism of integral varieties over $k$ whose generic fiber is a smooth curve of genus $\geq 2$.
For any integer $m \geq 1$, let $\Cc_\mathcal{B}^m$ be the $m$-th fiber product of $\Cc$ over $\mathcal{B}$.
Then, for $m$ sufficiently large, $\Cc_\mathcal{B}^m$ admits a dominant rational map to a variety of general type.
	\end{thm}
 
	Rectenly,  Dimitrov, Gao, and Habegger proved a excellent relation between the rank of a Jacobian and the number of rational points on the curve itself, using the technically machinery of Vojta's method and hight inequalities. 
	\begin{thm}[\cite{Dimitrov2021}]
		\label{dgh}
		Let $g \geq 1$ and $d \geq 1$ be integers.
There exists a constant $c=c(g,d)$ such that if $C$ is a smooth curve of genus $g$ defined over a number field $k$ with $[k:\Q] \leq d$, then
		\[ \#C(k) \leq c^{1+r(J_C(k))}.
\]
	\end{thm}

    \section{Geometric Constructions: Parameter Spaces}
    \label{sec_geom_constr}

    We construct the necessary parameter spaces, closely following the methods detailed in \cite{Hazama1991, Salami2019, Salami2025toappear}. Let $k$ be a number field containing $\zeta_s$, a primitive $s$-th root of unity, and fix integers $r \geq 1, s \geq 2$.

    \subsection{The Parameter Space $\Wc_n$ via Twists}
    Consider the universal curve $\Cc \subset \Pp^2_{(u,x,y)} \times_k \Pp^1_{(a,b)}$ defined by the homogenization of $y^s = x(ax^r+b)$. It carries an automorphism $[\zeta_s]: ([1:x:y], [a:b]) \mapsto ([1:x:\zeta_s y], [a:b])$. For $n \ge 0$, define the $(n+1)$-fold fiber product $\Vc_n = \Cc \times_{\Pp^1} \cdots \times_{\Pp^1} \Cc$. An element is $([a:b], (P_0, \dots, P_n))$ where $P_i=([1:x_i:y_i])$ lies on the curve $C_{a,b}: y^s = x(ax^r+b)$. Let $G = \langle \tau \rangle \cong \Z/s\Z$ act on $\Vc_n$ via $\tau = ([\zeta_s]_0, \dots, [\zeta_s]_n)$. The geometric parameter space is the quotient variety $\Wc_n = \Vc_n/G$.

    Let $\KK_n = k(\Wc_n)$ be the function field of $\Wc_n$, which is the fixed field $k(\Vc_n)^G$. Let $\Ll_n = k(\Vc_n)$. The field extension $\Ll_n/\KK_n$ is cyclic of degree $s$, generated by any $y_i$ satisfying $y_i^s = x_i(ax_i^r+b)$. Consider the twist $\widetilde{C}$ of the generic curve $C$ by this extension. As established in \cite{Hazama1991, Salami2019}, $\widetilde{C}$ is defined over the function field $\KK_n$ by the equation:
	\begin{equation} \label{twist_curve_eq}
		x_0(ax_0^r+b)y^s = x(ax^r+b).
	\end{equation}
	This twisted curve $\widetilde{C}$ possesses $n+1$ canonical $\KK_n$-rational points:
	\begin{equation} \label{twist_points_eq}
		\widetilde{P}_0 = (1:x_0:1), \quad \widetilde{P}_i = \left(1:x_i:\frac{y_i}{y_0}\right) \quad \text{for } i=1,\dots,n.
	\end{equation}
	Let $\widetilde{J}_C$ denote the Jacobian of $\widetilde{C}$. The points $\widetilde{P}_i$ induce points $\widetilde{Q}_i \in \widetilde{J}_C(\KK_n)$ by a chosen embedding $C \hookrightarrow J_C$.
	\begin{thm}[\cite{Salami2019, Salami2021Corrigendum}]
		\label{thm_rank_over_function_field}
		The points $\widetilde{Q}_0, \dots, \widetilde{Q}_n$ are linearly independent in $\widetilde{J}_C(\KK_n)$ over $\Z$. If $m_0 = \rk_\Z(\End_k(J_C)) \ge 1$, then the rank of the Mordell-Weil group satisfies 
		$\rk(\widetilde{J}_C(\KK_n)) =  m_0 (n+1)$.
	\end{thm}

    \subsection{The Complete Intersection  $\Xc_n$}
    We require a suitable open set of parameters. Let $\Sym^{n}(\Pp^1)$ denote the $n$-fold symmetric product of $\Pp^1$, isomorphic to $\Pp^{n}$ over $k$ \cite{Maakestad2005}. Define the hypersurface $\Hc_{r,n} \subset \Sym^{n+1}(\Pp^1_x)$ as the Zariski closure of the locus where
	$$ \prod_{0\leq i < j\leq n}^{} x_i x_j (x_i^r-x_j^r) = 0. $$
	Let $\Uc_{r,n} = \Sym^{n+1}(\Pp^1_x) \setminus \Hc_{r,n}$. A point $\ba_n$ represented by $\{[1:\alpha_0], \dots, [1:\alpha_n]\}$ lies in $\Uc_{r,n}(k)$ if and only if $\alpha_i \in k$ are distinct, non-zero, and satisfy $\alpha_i^r \neq \alpha_j^r$ for $i \neq j$.

    Consider the variety $\Xc_n \subset \Pp^n_X \times \Pp^n_Y$, defined by the $n-1$ bi-homogeneous equations detailed in \cite{Salami2025toappear} (where $Y_i$ corresponds to $y_i^s$):
	\begin{equation} \label{Xn_def_eq1}
		f_{i-1}: = X_{1}X_{i}(X_{i}^{r-1}-X_{1}^{r-1})Y_{0}^s+X_{0}X_{i}(X_{0}^{r-1}-X_{i}^{r-1})Y_{1}^s+X_{0}X_{1}(X_{1}^{r-1}-X_{0}^{r-1})Y_{i}^s=0
	\end{equation}
	for $i=2, \dots, n$. Equivalently, these conditions arise from the determinantal equations:
	\begin{equation} \label{Xn_def_eq2}
		\det \begin{pmatrix} X_0 & X_1 & X_i \\ X_0^r & X_1^r & X_i^r \\ Y_0^s & Y_1^s & Y_i^s \end{pmatrix} = 0, \quad \text{for } i=2, \dots, n.
	\end{equation}

Let us assume that $k$ contains a fixed $s$-th unit of unity  $\zeta_s$, and  a fixed   $r$-th unit of unity $\zeta_r$. 
 Then, the following trivial points belong to $\Xc_n(k)$:
$$( [\zeta_r^{j_0}:\dots: \zeta_r^{j_n}], [\zeta_s^{i_0}:\dots: \zeta_s^{i_n}]),    $$
where  $0\leq i_0, \cdots i_n,   j_0, \cdots j_n \leq d-1$ and  $d=lcm(r,s)$. 

	\begin{thm} \label{thm_birational_equiv}
		For $n \geq 2$, the varieties $\Wc_n$ and $\Xc_n$ are $k$-birational.
	\end{thm}
    \begin{proof} (See \cite[Thm 4.1]{Salami2025toappear} for the analogous construction).
        Let $Y_i$ represent $y_i^s$. The map $\varphi_n: \Wc_n \dashrightarrow \Xc_n$ is defined generically by
          $$([a:b], ([1:x_i:y_i])) \mapsto ([x_0:\dots:x_n], [y_0^s : y_1^s : \dots : y_n^s]).$$ 
          The inverse map $\varphi_n^{-1}: \Xc_n \dashrightarrow \Wc_n$ sends $[x_0:\dots:x_n:y_0:\dots:y_n]$ to the class corresponding to $([a:b], ([1:x_i: y_i))$, where $a$ and $b$ are determined by the first two points via $y_0^s=x_0(ax_0^r+b)$ and $y_1^s=x_1(a x_1^r+b)$. Explicitly:
        \[ a = \frac{x_1 y_0^s - x_0 y_1^s}{D}, \quad b = \frac{x_0^r y_1^s - x_1^r y_0^s}{D}, \]
        where $D=x_0 x_1(x_0^{r-1} -x_1^{r-1}) \neq 0.$
        These maps are well-defined rational maps, inverse to each other on dense open subsets where the denominators (related to the conditions defining $\Uc_{r,n}$) are non-zero.
	\end{proof}

    Let $\pi_n: \Wc_n \to \Sym^{n+1}(\Pp^1_x)$ be the natural projection. For $\ba_n \in \Uc_{r,n}(k)$, let $\Wc_{\ba_n} = \pi_n^{-1}(\ba_n)$ be the fiber over $\ba_n=\{[1:\alpha_0], \dots, [1:\alpha_n]\}$.
    \begin{prop} \label{prop_point_curve_corr} 
		There exists a dense open subset $\Uc_{\ba_n} \subset \Wc_{\ba_n}$ such that there is a one-to-one correspondence between the set $\Cc_{\ba_n}(k)$ of smooth curves $C: y^s = x(ax^r+b)$ passing through points $P_i$ with $x$-coordinates $x(P_i)=\alpha_i$, and the set of $k$-rational points $\Uc_{\ba_n}(k)$.
	\end{prop}
    \begin{proof}
        A smooth curve $C \in \Cc_{\ba_n}(k)$ determined by $[a:b]$ yields points $P_i=(\alpha_i, \beta_i)$ and defines a point in $\Wc_{\ba_n}(k)$. The smoothness condition ensures $ab \neq 0$ and avoids singular points, defining the open set $\Uc_{\ba_n}$. Conversely, a point in $\Uc_{\ba_n}(k)$ corresponds to parameters $[a:b]$ defining a smooth curve $C \in \Cc_{\ba_n}(k)$. The map is well-defined up to isomorphism respecting the points.
    \end{proof}
    
    A point $\ba_n \in \Uc_{r,n}(k)$ is generic if the specialization map for the Jacobian $\widetilde{J}_C$ preserves the rank $\rk(\widetilde{J}_C(\KK_n))$. Silverman's specialization theorem \cite{Silverman1983} implies this holds on a Zariski open dense subset of $\Uc_{r,n}$.
    \begin{prop} \label{prop_rank_specialization_open} (Cf. \cite[Prop 3.3]{Salami2025toappear})
		The set of points $\ba_n \in \Uc_{r,n}(k)$ for which there exists a curve $C \in \Cc_{\ba_n}(k)$ such that the rational points $(\alpha_i, \beta_i)$ induce linearly independent points in $J_C(k)$ contains a Zariski open dense subset of $\Uc_{r,n}$.
	\end{prop}

    \section{Geometry and Rational Points of Fiber Curves}
    \label{sec_fiber_analysis}

    We analyze the geometry of the fiber curve $\Xc_{\ba_n}$, which, by Theorem \ref{thm_birational_equiv} and Proposition \ref{prop_point_curve_corr}, parameterizes the curves in   $\Cc_{\ba_n}(k)$. Let $g_{1,n}: \Xc_n \to \Pp^n_X$ be the projection. For $\ba_n  \in \Uc_{r,n}(k)$, the fiber $\Xc_{\ba_n} = g_{1,n}^{-1}(\ba_n)$ is a curve in $\Pp^n_Y$ defined by equations \eqref{Xn_def_eq1} specialized at $X_i = \alpha_i$. We note that we use the fact that $\Sym^{n}(\Pp^1)$ is  isomorphic to $\Pp^{n}$ over $k$ as given in \cite[Cor. 2.6]{Maakestad2005}.

	\begin{prop} \label{prop_fiber_geometry_main}
	Keeping the above notations in mind, we have the followings.
		\begin{enumerate}
			\item The fiber $\Xc_{\ba_n}$ is a smooth, projective, complete intersection curve over $k$. Its genus is given by
			\[ g(\Xc_{\ba_n}) = 1 + \frac{s^{n-1}((n-1)s - n-1)}{2}. \]
			\item The $k$-gonality $\gamma_k(\Xc_{\ba_n})$ satisfies the lower bound $\gamma_k(\Xc_{\ba_n}) \geq (s-1)s^{n-2}$ for $n \ge 2$.
		\end{enumerate}
	\end{prop}
    \begin{proof}
        (1) For $\ba_n \in \Uc_{r,n}(k)$, the specialized Jacobian matrix of the defining equations \eqref{Xn_def_eq1} (viewed as homogeneous polynomials of degree $s$ in $Y_0, \dots, Y_n$) has maximal rank $n-1$. Thus, $\Xc_{\ba_n}$ is a smooth complete intersection of $n-1$ hypersurfaces of degree $s$ in $\Pp^n_Y$. The genus formula follows from \cite[Ex. II.8.4]{Ha} having $d=s^{n-1}$ and the degree of the canonical divisor 
         $$2g(\Xc_{\ba_n})-2=\deg(K_{\Xc_{\ba_n}}) = d(\sum \deg(f_i) - (n+1)) = s^{n-1}((n-1)s - n-1).$$

        (2) The $k$-gonality $\gamma_k( \Xc_{\ba_n})$ is the minimal degree of a dominant $k$-rational map $\Xc_{\ba_n} \to \Pp^1$. By Lazarsfeld's Theorem  on the gonality of complete intersection curves (see \cite{Lazarsfeld1997} or \cite[Thm 2.7]{Salami2025toappear}) applied to the complete intersection $\Xc_{\ba_n} \subset \Pp^n_Y$ defined by $n-1$ forms of degree $d_1=\dots=d_{n-1}=s \ge 2$, we have
         $$\gamma_k(\Xc_{\ba_n}) \geq (d_1-1)d_2 \cdots d_{n-1} = (s-1)s^{n-2}.$$
	\end{proof}

    Now, analogous to \cite[Thm 4.5]{Salami2025toappear}),  we have the  following structural results on  
     $\Xc_{\ba_n}(k)$, based on the genus calculation.
	\begin{thm} \label{thm_low_genus_fibers}  
			Given $\ba_n \in \Uc_{r,n}(k)$, we have the followings:
		\begin{enumerate}
			\item Case $s=2, n=2$: The genus is $g(\Xc_{\ba_2})=0$, hence $\Xc_{\ba_2}$ is a conic in $\Pp^2$ with $k$-rational points; thus $\Xc_{\ba_2}(k)$ is infinite.
			\item Case $s=2, n=3$: The genus is $g(\Xc_{\ba_3})=1$, hence $\Xc_{\ba_3}$ is an elliptic curve over $k$; thus $\Xc_{\ba_3}(k)$ can be infinite.
			\item Case $s=3, n=2$: The genus is $g(\Xc_{\ba_2})=1$, hence $\Xc_{\ba_2}$ is an elliptic curve over $k$; thus $\Xc_{\ba_2}(k)$ can be infinite.
            \item For all other cases $(s,n)$ with $n \ge 2$, the genus $g(\Xc_{\ba_n}) \ge 2$ hence  $\Xc_{\ba_2}(k)$ is finite by Falting's Theorem.
		\end{enumerate}
	\end{thm}
    For the high genus cases ($g \ge 2$), we employ the theory of towers of curves \cite{Xarles2013}. For a fixed $\ba = \lim \ba_n$, consider the tower $\Cb_{\ba} = (\{\Xc_{\ba_m}\}_{m \ge n_0}, \{\phi_m\}_{m \ge n_0})$ where $n_0$ ensures $g(\Xc_{\ba_{n_0}}) \ge 2$, and $\phi_{m+1}: \Xc_{\ba_{m+1}} \to \Xc_{\ba_m}$ is the projection forgetting the last coordinate $Y_{m+1}$.

	\begin{thm} \label{thm_tower_finiteness} 
		Let $k$ be a number field. Let $n_0$ be such that $g(\Xc_{\ba_{n_0}}) \ge 2$. For any fixed $\ba = \lim \ba_n$ where $\ba_m \in \Uc_{r,m}(k)$ for all $m$, the tower $\Cb_{\ba} = \{\Xc_{\ba_m}\}_{m \ge n_0}$ has infinite gonality. Consequently, for any sufficiently large $m$ (depending on $\ba$), the set $\Xc_{\ba_m}(k)$ consists only of trivial points with respect to the tower $\Cb_{\ba}$.
	\end{thm}
	\begin{proof}
		By Proposition \ref{prop_fiber_geometry_main}(2), $\gamma_k(\Xc_{\ba_m}) \ge (s-1)s^{m-2} \to \infty$ as $m \to \infty$. Thus, the tower $\Cb_{\ba}$ has infinite gonality. The conclusion then follows directly from the main results of Xarles in \cite{Xarles2013} applied with $d=1$.
	\end{proof}
	
    \section{Proofs of Finiteness and Uniformity Theorems}
    \label{sec_finiteness_uniformity_proofs}

    We now prove Theorems \ref{finiteness_theorem} and \ref{uniformity_theorem}, using the properties of the fibers $\Xc_{\ba_n}$ and standard results and conjectures provided in Section~\ref{res-conjs}. Recall the correspondence  given in Proposition \ref{prop_point_curve_corr}). The points in $\Cc_{\ba_n}(k)$ correspond bijectively to points in an open subset $\Uc_{\ba_n}(k) \subset \Xc_{\ba_n}(k)$. Trivial points in $\Xc_{\ba_n}(k)$ (in the sense of the tower) often correspond to degenerate curves or curves independent of some $\alpha_i$, and are typically excluded from $\Uc_{\ba_n}$.

    \begin{proof}[Proof of Theorem \ref{finiteness_theorem}]
		Let $n_0=4$ if $s=2$ and $n_0=3$ otherwise.
        If $n < n_0$, then $g(\Xc_{\ba_n}) \le 1$ by Theorem \ref{thm_low_genus_fibers}. For suitable $\ba_n$, $\Xc_{\ba_n}(k)$ can be infinite (either $\Pp^1(k)$ or an elliptic curve of positive rank), by  \cite[Lemas 4.3 and 4.4]{Salami2025toappear}.
         This allows $\Cc_{\ba_n}(k)$ to be infinite.

        If $n \geq n_0$, then $g(\Xc_{\ba_n}) \geq 2$ by Theorem \ref{thm_low_genus_fibers}(4). Hence, the Faltings' Theorem implies that $\Xc_{\ba_n}(k)$ is finite for any given $\ba_n \in \Uc_{r,n}(k)$. Since $\Cc_{\ba_n}(k)$ corresponds to a subset $\Uc_{\ba_n}(k) \subset \Xc_{\ba_n}(k)$, $\Cc_{\ba_n}(k)$ must also be finite.
        The boundedness of rank for $n \ge n_0$ follows from the finiteness of $\Cc_{\ba_n}(k)$: there is a maximum rank among the finitely many Jacobians $J_C$ for $C \in \Cc_{\ba_n}(k)$.
	\end{proof}

    \begin{proof}[Proof of Theorem \ref{uniformity_theorem}]
		Assume Weak Lang holds (Conjecture \ref{conj1}). Let $n \ge n_0$, so $g = g(\Xc_{\ba_n}) \geq 2$. The curves $\Xc_{\ba_n}$ form a family $\Xc_n \to \Uc_{r,n}$ over $k$, where the base $\Uc_{r,n}$ is an open subset of $\Sym^{n+1}(\Pp^1) \cong \Pp^{n+1}$. By Uniformity I (Theorem \ref{UB}), there exists a bound $N(k,g)$ depending only on $k$ and the genus $g = g(\Xc_{\ba_n})$, such that $\#\Xc_{\ba_n}(k) \leq N(k,g)$ for all $\ba_n \in \Uc_{r,n}(k)$. Since $\#\Cc_{\ba_n}(k) \le \#\Xc_{\ba_n}(k)$, the bound $M_0 = N(k,g)$ suffices. Note that $g$ depends on $n, s$.

		For the second part, assume $g = g_{r,s}(C) \geq 2$. Consider the family $f: \mathcal{C} \to \Pp^1_{(a,b)}$ for the original curves $y^s=x(ax^r+b)$. The generic fiber has genus $g\ge 2$. By the Correlation Theorem \ref{corel}, for $m$ large enough, the $(m+1)$-fold fiber product $\mathcal{C}^{m+1}_{\Pp^1}$ maps dominantly to a variety $W$ of general type. A curve $C \in \Cc_{\ba_m}(k)$ corresponds to a point $[a:b]$ and $m+1$ points $(\alpha_i, \beta_i)$ on $C_{a,b}$. This yields a $k$-rational point on $\mathcal{C}^{m+1}_{\Pp^1}$, hence on $W$. By Weak Lang (Conjecture \ref{conj1}), $W(k)$ is not Zariski dense. The uniformity theorem derived from Weak Lang (Theorem \ref{UB}), applied to the fibers $C_{a,b}$, states there is a bound $N(k, g)$ such that $\#C_{a,b}(k) \le N(k, g)$. If we choose $m_0$ such that $m_0+1 > N(k, g)$, then no curve can have $m_0+1$ points $(\alpha_i, \beta_i)$. 
		Therefore, $\Cc_{\ba_m}(k)$ must be empty for all $m \ge m_0$.
        If Strong Lang holds, Theorem \ref{UGB} gives $N(g)$ independent of $k$, making $m_0$ independent of $k$.
	\end{proof}

    \section{Proof of the Main Theorem}
    \label{sec_main_theorem_proof}

    The proof combines the uniformity established in Theorem \ref{uniformity_theorem} with the crucial link between point counts and rank.

    \begin{proof}[Proof of Theorem \ref{main_theorem}]
		Let $k$ be a number field, $[k:\Q]=d$. Let $\mathcal{F}$ be the family of smooth curves $C: y^s=x(ax^r+b)$ over $k$ with genus $g=g_{r,s}(C) \geq 1$. Assume, for contradiction, that the rank $r(J_C(k))$ is unbounded for $C \in \mathcal{F}$.
        Then there exists a sequence of non-isomorphic curves $C_j \in \mathcal{F}$ such that $r_j = r(J_{C_j}(k)) \to \infty$ as $j \to \infty$.
        By the result of Dimitrov, Gao, and Habegger (Theorem \ref{dgh}), there is a constant $c=c(g,d)$ such that $\#C_j(k) \leq c^{1+r_j}$. Since $r_j \to \infty$, it follows that $\#C_j(k) \to \infty$.

        Now, assume the Strong Lang Conjecture (Conjecture \ref{conj1}). This implies the strong version of Theorem \ref{uniformity_theorem}: there exists an integer $m_0$ (depending only on $r, s$, possibly $g$, but not $k$) such that for any $m \geq m_0$, $\Cc_{\ba_m}(K) = \emptyset$ for any number field $K$ and any valid $\ba_m \in \Uc_{r,m}(K)$. This means no smooth curve in the family \eqref{curve_main} over any number field $K$ can possess $m+1$ points $(\alpha_i, \beta_i)$ where the $\alpha_i$ are distinct, non-zero, and satisfy $\alpha_i^r \neq \alpha_j^r$, provided $m \ge m_0$.

        Since $\#C_j(k) \to \infty$, we can choose an index $j_0$ large enough such that $\#C_{j_0}(k) > s \cdot (m_0+1) + 1$. The curve $C_{j_0}$ has at most one point with $x=0$ (namely $(0,0)$). For any non-zero $x$-coordinate $\alpha$, there are at most $s$ corresponding $y$-coordinates $\beta$. Thus, the number of distinct non-zero $x$-coordinates among points in $C_{j_0}(k)$ is at least $\lceil (\#C_{j_0}(k)-1)/s \rceil$. Since $\#C_{j_0}(k) \to \infty$, this quantity eventually exceeds $m_0+1$.
        Therefore, we can select $m_0+1$ points $P_0, \dots, P_{m_0}$ from $C_{j_0}(k)$ whose $x$-coordinates $\alpha_0, \dots, \alpha_{m_0}$ are distinct and non-zero. Let $\ba_{m_0} = \{\alpha_0, \dots, \alpha_{m_0}\}$. Since $C_{j_0}$ is smooth, its parameters $[a:b]$ are such that the discriminant conditions are met, which typically ensures $\alpha_i^r \neq \alpha_{j_0}^r$ if $\alpha_i \neq \alpha_{j_0}$ (unless $r=0$, but $r \ge 1$). Thus, $\ba_{m_0}$ is a valid configuration in $\Uc_{r, m_0}(k)$.
        By definition, the curve $C_{j_0}$ belongs to the set $\Cc_{\ba_{m_0}}(k)$.
        However, since $m_0 \geq m_0$, the strong uniformity result implies $\Cc_{\ba_{m_0}}(k)$ must be empty. This is a contradiction.
        The initial assumption that the rank is unbounded must be false. Therefore, $r(J_C(k))$ is uniformly bounded for $C \in \mathcal{F}$.
	\end{proof}
\section{Application  of the Results on Elliptic Curves}
\label{sec:applications}

The general family of curves $C_{a,b}: y^s = x(ax^r + b)$ investigated herein contains an important family of elliptic curves over $\Q$,  containing at least one 2-torsion point $(0,0).$
 By specializing our results, we provide new geometric evidence for uniformity conjectures related to these curves. We focus on the case $s=2$ and $r=2$, which corresponds to the one-parameter family $E_B: y^2 = x(x^2 + B)$. The largest know rank in this family is 14, due to M. Watkins, according to \cite{DujellaRankRecord}.  We will consider this curve in the next section.

A central case of interest in this family is the congruent number elliptic curve $E_N: y^2 = x^3 - N^2x$ over $k=\Q$.  A positive integer $N$ is a congruent number if and only if $E_N(\mathbb{Q})$ has positive rank.  The largest know rank congruent number elliptic curve is 7. Note that there are only 27 case of known congruent number elliptic curve of rank 7, by the recent huge computational work \cite[Table 5]{Watkins2014}.
But, the question of whether these ranks are (uniformly) bounded as $N$ varies is a major open problem.
Our Theorem \ref{main_theorem} provides a direct conditional answer.

\begin{thm} 
	\label{CN-rank1}
	Assume the strong version of Lang's conjecture (Conjecture \ref{conj1})  holds. Then the Mordell-Weil ranks 
$r(E_B(\Q))$	  of elliptic curves $E_B: y^2=x (x^2+B)$  are uniformly bounded as $B$ varies over all integers.
\end{thm}

Theorems \ref{finiteness_theorem} and \ref{uniformity_theorem} apply directly to this family for a \emph{fixed} set of points $a_n = \{\alpha_0, \dots, \alpha_n\}$.

\begin{thm} 
	\label{CN-rank2}
For any fixed  $n  \ge 4$ and  distinct non-zero  rational numbers $\alpha_0, \dots, \alpha_n$, the set of elliptic curves $E_B: y^2=x (x^2+B)$, in particular  congruent number curves,  having $(n+1)$   points with $x$-coordinates $\alpha_i$'s  is finite;  and hence their rank is bounded assuming Weak Lang's Conjecture.
\end{thm}

 Note that the second part of Theorem \ref{uniformity_theorem}, which establishes emptiness, does not apply, as that part requires the generic curve $C$ to have genus $g \ge 2$, whereas elliptic curves have $g=1$.

\section{Examples of elliptic curve of the form $y^2=x(x^2+B)$ over $\Q$}
\label{examples}
In this section, we provide a concrete example to show how specific  high  rank  elliptic curve of the form 
$y^2=x(x^2+B)$  with a known set of rational points correspond to a single rational point on the high-genus fiber variety $\Xc_{\ba_n}$.

\subsection{Watkins' rank 14 elliptic curve curve}

In 2002, according to \cite{DujellaRankRecord},  M. Watkins discovered the following elliptic curve of rank 14:
\[ E: y^2 = x(x^2+B), \quad B= 402599774387690701016910427272483, \]
with 14   independent points $P_0, \dots, P_{13} \in E(\Q)$  as follows:	
{\small
	\begin{align*}
			P_{1} &= \left[ 17715373576525779,  3562569314711466369088086\right]; \\
			P_{2} &= \left[2626434695669379,  1037072601415883504491614\right]; \\
			P_{3} &= \left[2569230493256067,  1025344316293086716196318\right]; \\
			P_{4} &= \left[235538747268099,  307962520197557881526046\right]; \\
			P_{5} &= \left[72777729441003372,  20366017444893924849237282\right]; \\
			P_{6} &= \left[208383733118864688,  95565185470960061947766676\right]; \\
			P_{7} &= \left[36178079522739,  120686925577870348570566\right]; \\
			P_{8} &= \left[103189419061250643,  33768487838255557704513174\right]; \\
			P_{9} &= \left[1751414347117072176,  2317991574180462284959749972\right]; \\
			P_{10} &= \left[\frac{306104494367228425}{4},  \frac{175082211930567255911081155}{8}\right]; \\
			P_{11} &= \left[\frac{54693351931994304}{25},  \frac{118007688830447299097189592}{125}\right]; \\
			P_{12} &= \left[2696555916804876,  1051304226981395145047478\right]; \\
			P_{13} &= \left[2842774711299072,  1080497092155012281695968\right]; \\
			P_{14} &= \left[\frac{46439279877409015377}{1681},  \frac{391130341466321391183789029622}{68921}\right];
	\end{align*}
Let $\alpha_i = x(P_i)$ be their $x$-coordinates. We consider the point $\ba_{13} = (\alpha_0, \dots, \alpha_{13})$. This single elliptic curve corresponds to a single rational point
$[y(P_0): \cdots : y(P_{13})]$ 
on the fiber curve $\Xc_{\ba_{13}}$ given by
the following equations,

{\small
\begin{align*}
c Y_2^2 &= -2005574349751504644509255068888546147636036509227512033844736 Y_0^2 \\
	&\quad + 13983695978950970650784060349998560435900132832925220041088083936 Y_1^2 \\[1em]
	c Y_3^2 &= -4233067983346208831400277563796271086573328385165098610128640 Y_0^2 \\
	&\quad + 1309292036088229411960363759074835291233991638699966377191407840 Y_1^2 \\[1em]
	c Y_4^2 &= 1011104702106959966985375105340722510931960172621524731500814981084 Y_0^2 \\
	&\quad - 6424200495110553336288170238275159352388183359062410095727794889884 Y_1^2 \\[1em]
	c Y_5^2 &= 23762331636106083769294103073652373057853327275984504065984826524656 Y_0^2 \\
	&\quad - 159144489537926914861651403395501312773232204001959007351534004911856 Y_1^2 \\[1em]
	c Y_6^2 &= -655334329584817993937527781597442724834488395267414206545120 Y_0^2 \\
	&\quad + 201138238739636920084978130605027404969680313391459877617264320 Y_1^2 \\[1em]
	c Y_7^2 &= 2883969531779178007119996834526013643664550288687253012258879742176 Y_0^2 \\
	&\quad - 18891361499132686442750197016993794427816316393581432418549425477376 Y_1^2 \\[1em]
	c Y_8^2 &= 14110173023443128255109375776390154645419757128450875056516540285881840 Y_0^2 \\
	&\quad - 95163978158604984065119817852774342392762125208723452425550361728461040 Y_1^2 \\[1em]
	c Y_9^2 &= \frac{75242611761356859253235481398487232838754976738722113808659465139675}{64} Y_0^2 \\
	&\quad - \frac{480882418076579419966944014369610591537290323036447925199180755829675}{64} Y_1^2 \\[1em]
	c Y_{10}^2 &= - \frac{189614140868852018740282838881310011367959190350103787254334894144}{15625} Y_0^2 \\
	&\quad + \frac{187150591634919392020263417408058040486186269756341823308987007124544}{15625} Y_1^2 \\[1em]
	c Y_{11}^2 &= 2643511524719788541214295480135428464498809650075991037138940 Y_0^2 \\
	&\quad + 14644668644468681014477088061339920304870932259664196848994881860 Y_1^2 \\[1em]
	c Y_{12}^2 &= 8834265743370842136672317242120812904001739441027169010064384 Y_0^2 \\
	&\quad + 15397975254195152920151011006662938909233294188333054350195106816 Y_1^2 \\[1em]
	c Y_{13}^2 &= \frac{260663250472759556438333412021381491546525306445523460817721558096121211624}{4750104241} Y_0^2 \\
	&\quad - \frac{1044638835614736099507956192081287140190414011074264547599984407729652564424}{4750104241} Y_1^2
	\end{align*}}
where $c=	-14281215675385850918697452819453138714196110374224753373383199200$.

According to Proposition \ref{prop_fiber_geometry_main}, the genus of this  curve is:
\[ g(\Xc_{\ba_{13}}) = 1 + \frac{2^{12}(2 (13-1) -13-1)}{2} =   20481. \]

Therefore, this dramatically illustrates the principle of the construction: information about a set of points on one curve is transformed into information about a single point on a much more complex curve with genus $20481$.

\subsection{Roger's rank 7 congruent number elliptic curve }
In this  subsection, we consider the smallest rank 7 congruent number know 
 by N.F. Rogers in \cite{Rogers2000Rank}, namely, 
 $$E: y^2= x (x^2- (797507543735)^2),$$
 with seven linearly independent points:
 \begin{align*}
 	P_1 &= [2349199039600, 3386809128504045300], \\
 	P_2 &= \left[ \frac{41883387252225625}{50176}, \frac{2531093164311743323699875}{11239424} \right], \\
 	P_3 &= \left[ \frac{6509060981758225}{1764}, \frac{512730950467913482286575}{74088} \right], \\
 	P_4 &= \left[ \frac{30950527816902400}{27889}, \frac{3786502295899040518760400}{4657463} \right], \\
 	P_5 &= \left[ \frac{-20042809470080964}{116281}, \frac{12818431302547365397586334}{39651821} \right], \\
 	P_6 &= \left[ \frac{2829381632947105879686876}{212780569}, \frac{4759236582163811059977032999691297174}{3103830160003} \right], \\
 	P_7 &= \left[ \frac{-5460724565156956552975}{23125893184}, \frac{1301772971778523141805652666878775}{3516800828277248} \right].
 \end{align*}
We let $\alpha_i = x(P_i)$ be their $x$-coordinates and consider the point $\ba_{6} = (\alpha_0, \dots, \alpha_{16})$. This single elliptic curve corresponds to a single rational point
$[y(P_0): \cdots : y(P_{6})]$ 
on the fiber curve $\Xc_{\ba_{6}}$ given by
the following equations:
{\small 
\begin{align*}
	Y_2^2 &= \frac{513595127541878181175302438231511}{122044679535105832273033658831424} Y_0^2 
	- \frac{206341385505716449378970894336}{27798077518017910047611529435} Y_1^2 \\[2ex]
	Y_3^2 &= \frac{3180457837001802166902279136112}{60699325443314568164843087704167} Y_0^2 
	+ \frac{358745519042118086318415587311616}{303496627216572840824215438520835} Y_1^2 \\[2ex]
	Y_4^2 &= \frac{760976582596215351600983126461880167}{74971796963450955717129323249442937500} Y_0^2 
	- \frac{110142131332869225026801997925436096512}{468573731021568473232058270309018359375} Y_1^2 \\[2ex]
	Y_5^2 &= \frac{19406883587924920002928615522332376834620074078564610239}{93501974059845673545423496982841643371937500} Y_0^2 \\
	& \quad - \frac{341357823067352114276128561961468335530766054960979247104}{584387337874035459658896856142760271074609375} Y_1^2 \\[2ex]
	Y_6^2 &= \frac{43602457892743533098592017055123115417418628899734753}{3263128278387778833488311732199040036764917844062460928} Y_0^2 \\
	& \quad - \frac{1021282358455283550237470713488177637062931484250112}{3186648709363065267078429425975625035903240082092247} Y_1^2
\end{align*}
}
By Proposition \ref{prop_fiber_geometry_main}, the genus of this  curve is  $g(\Xc_{\ba_{6}})=49$

\bibliographystyle{amsplain}
\bibliography{BMW-bib-2}

\end{document}